\title{A note on random walks in a hypercube}
\author{Stanislav Volkov\thanks{
                  University of Bristol}
        \and
        Timothy Wong\thanks{
                      University of Bristol
                  }
        }
\documentclass[12pt]{article}
 
\usepackage{graphicx}
\usepackage{amsmath,amssymb}
\def\P{{\mathbb{P}}}
\def\DD{\displaystyle}
\newtheorem{theorem}{Theorem}

\newtheorem{remark}{Remark}   %% And a not so common one.

\newenvironment{proof}{{\sc Proof:}}{~\hfill QED}
\newenvironment{AMS}{}{}
\newenvironment{keywords}{}{}

\begin{document}
\newpage
\maketitle
%%%%%%%%%%%%%%%%%%%%%%%%%%%
% abstract, keywords and Subject classification are optional.
%%%%%%%%%%%%%%%%%%%%%%%%%%%
\begin{abstract}
    We study a simple random walk on an $n$-dimensional hypercube.
    For any starting position we find the probability of hitting
    vertex $a$ before hitting vertex $b$, whenever $a$ and $b$ share the
    same edge. This generalizes the model in~\cite{DS} (see
    Exercise~1.3.7 there).
\end{abstract}

% Most people don't use these, so they are "commented out"
% by starting the lines with a "%"


Consider an $n$-dimensional hypercube, that is a graph with $2^n$
vertices in the set $\{0,1\}^n$. A vertex $x$ of a hypercube can
be encoded by a sequence $x=(x_1,x_2,\dots,x_n)$ where each $x_i$
is either $0$ or $1$.

Two vertices $x=(x_1,x_2,\dots,x_n)$ and $y=(y_1,y_2,\dots,y_n)$
are connected by an edge, if and only if $\sum_{i=1}^n
|x_i-y_i|=1$, that is $x$ and $y$ differ in exactly one coordinate
(for example when $n=5$, $x=(0,1,0,1,0)$, and $y=(0,0,0,1,0)$\,).
For two vertices $x$ and $y$ the (graph) distance between them is
the quantity $|x-y|:=\sum_{i=1}^n |x_i-y_i|$, that is the smallest
number of edges on the path connecting $x$ and $y$.

A simple random walk on a hypercube is a particle which moves from
one vertex to another along the edges of this graph, with equal
probabilities. Since each vertex is connected by an edge to
exactly $n$ other vertices, the probability to go to any
particular neighbor equals $1/n$, and is independent of the past
movements. Such a walk has been fairly extensively studied,
especially its asymptotic properties, but we were not able to find
in the published materials the exact formula found by us in
Theorem~\ref{ContThrm}. Probably the most relevant references for
the walk on the hypercube would be~\cite{Diac}, \cite{DS}
and~\cite{Voit}, which of course does not expire the set of the
available literature on the topic.

Suppose we have two distinct vertices $a$ and $b$. Start a random
walk at some point $X_0=x$ and denote its position at time $n$
 as $X_n$. Assume that the walk stops when it hits
either $a$ or $b$. Our aim is to compute the probability that
$X_n$ hits $b$ before it hits $a$. Formally, if
$$
 \tau=\inf\{n\ge 0:\ X_n=a\mbox{ or }b\}
$$
we want to compute
$$
 p_{a,b}(x)=\P(X_{\tau}=b \ |\ X_0=x).
$$

Though we were not able to answer this question in general, we can
do it nonetheless in the case when $a$ and $b$ are immediate
neighbors, that is, connected by an edge. Without loss of
generality, assume from now on that
$$
a=(0,0,\dots,0,0)
$$
and
$$
b=(0,0,\dots,0,1).
$$
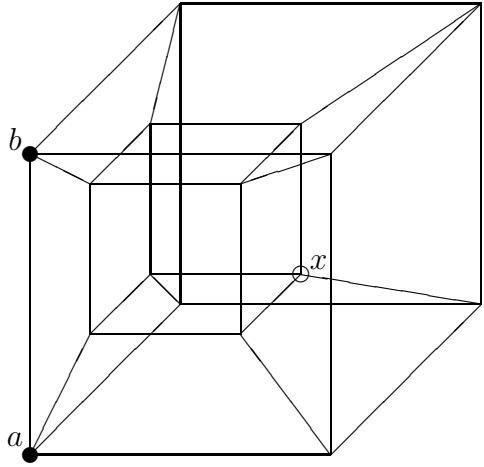
\begin{figure}[hbt]
     \centering
     \unitlength=0.4mm
     \begin{picture}(160.00,160.00)
          \put(0,0){\line(1,0){100.00}}
          \put(0,0){\line(0,1){100.00}}
          \put(100,100){\line(-1,0){100.00}}
          \put(100,100){\line(0,-1){100.00}}
          \put(50,50){\line(1,0){100.00}}
          \put(50,50){\line(0,1){100.00}}
          \put(150,150){\line(-1,0){100.00}}
          \put(150,150){\line(0,-1){100.00}}
          \put(0,0){\line(1,1){50}}
          \put(0,100){\line(1,1){50}}
          \put(100,0){\line(1,1){50}}
          \put(100,100){\line(1,1){50}}
%%%
          \put(20,40){\line(1,0){50}}
          \put(20,40){\line(0,1){50}}
          \put(70,90){\line(-1,0){50}}
          \put(70,90){\line(0,-1){50}}
          \put(40,60){\line(1,0){50}}
          \put(40,60){\line(0,1){50}}
          \put(90,110){\line(-1,0){50}}
          \put(90,110){\line(0,-1){50}}
          \put(20,40){\line(1,1){20}}
          \put(20,90){\line(1,1){20}}
          \put(70,40){\line(1,1){20}}
          \put(70,90){\line(1,1){20}}

%%%%
%

          \put(0,0){\line(1,2){20}}
          \put(0,100){\line(2,-1){20}}
          \put(100,0){\line(-3,4){30}}
          \put(100,100){\line(-3,-1){30}}
          \put(50,50){\line(-1,1){10}}
          \put(50,150){\line(-1,-4){10}}
          \put(150,50){\line(-6,1){60}}
          \put(150,150){\line(-3,-2){60}}

          \put(0.00,0.00){\circle*{5.00}}
          \put(0.00,100.00){\circle*{5.00}}
          \put(90.00,60.00){\circle{5.00}}
          \put(-5.00,5.00){\makebox(0,0)[cc]{$a$}}
          \put(-5.00,105.00){\makebox(0,0)[cc]{$b$}}
          \put(96.00,64.00){\makebox(0,0)[cc]{$x$}}
     \end{picture}
     \caption{A $4$-dimensional hypercube. $a=(0,0,0,0)$,
     $b=(0,0,0,1)$ and $x=(1,1,1,0)$.
     \label{SelfdualFig}}
\end{figure}
See Figure~\ref{SelfdualFig} for a possible location of point $x$.

\begin{theorem}\label{ContThrm}
    Suppose we start a simple random walk on the hypercube from point
    $x=(x_1,x_2,\dots,x_n)$. Then the probability that this walk
    hits $b$ before $a$, is
    \begin{eqnarray*}\label{SplitFunc}
         p(x)=p_{a,b}(x) = \left\{ \begin{array}{rl}
                          \frac {\DD 1}{\DD 2} -\frac{\DD \sum_{i=k+1}^{n} {n\choose i}}{\DD 2(2^n-1){n-1\choose k}},
                              & \mbox{~if $x_n=0$;} \\ \\
                         \frac {\DD 1}{\DD 2} +
                         \frac{\DD \sum_{i=k+1}^{n} {n\choose i}}{\DD 2(2^n-1){n-1\choose k}},
                             & \mbox{~if $x_n=1$},
                         \end{array} \right.
   \end{eqnarray*}
   where $k=x_1+x_2+\dots+x_{n-1}$.
\end{theorem}
\begin{proof}
As it is noted in~\cite{DS}, instead of finding the probability
$p(x)$ we can solve a seemingly completely different problem from
electric networks as follows. Suppose that each edge of the
hypercube is replaced by a unit resistor. Attach a 1 volt battery
to points $a$ and $b$, such that the voltage at $a$, denoted as
$v(a)$, equals $0$, and the voltage at $b$ is $v(b)=1$. Then the
voltage at vertex $x$, denoted by $v(x)$ is equal exactly to  the
unknown probability $p(x)$.

Thus, to solve our problem, we will use the language and the
methods borrowed from electric network theory. For example, if two
or more resistors are connected in series, they may be replaced by
a single resistor whose resistance is the sum of their
resistances. Also, $m$ ($\ge 2$) resistors in parallel may be
replaced by a single resistor with resistance $R$ equal to
$1/(R_1^{-1}+\dots+R_m^{-1})$ where $R_1,\dots,R_m$ are the
resistances of the original resistors.

From the symmetry it follows that the voltage at all vertices of
the cube lying in the set
$$
W_k:=\{x=(x_1,x_2,\dots,x_{n-1},0) \mbox{ such that }
 x_1+x_2+\dots+x_{n-1}=k \}
$$
is the same and depends on $k$ only. Let us denote this voltage
$w_k$. Similarly, the voltage at all vertices in
$$
W_k:=\{x=(x_1,x_2,\dots,x_{n-1},1) \mbox{ such that }
 x_1+x_2+\dots+x_{n-1}=k \}
$$
is also the same, let us denote this voltage $\tilde w_k$.
Moreover, from symmetry, the probability to hit $a=(0,\dots,0,0)$
before $b=(0,\dots,0,1)$ starting from $(x_1,x_2,\dots,x_{n-1},0)$
is the same as the probability to hit $b$ before $a$ starting from
$(x_1,x_2,\dots,x_{n-1},1)$, we obtain
\begin{equation}\label{eqtil}
\tilde w_k=1-w_k, \mbox{ for } k=0,1,\dots,n-1.
\end{equation}

Obviously, we have
\begin{equation}\label{eq0}
w_0=0.
\end{equation}
Now for $1\le k \le n-1$, the vertex
$x=(x_1,x_2,\dots,x_{n-1},0)\in W_k$ is connected to $n$ vertices
\begin{eqnarray*}
 &(1-x_1&,x_2,\dots,x_{n-1},0)\\
 (x_1,&1- x_2&,\dots,x_{n-1},0)\\
 &\dots&\\
 (x_1,x_2,\dots,&1-x_{n-1}&,0)\\
 (x_1,x_2,\dots,x_{n-1},&1)&\\
\end{eqnarray*}
where the last one lies in $\tilde W_k$, and among the first $n-1$
vertices $k$ lie in $W_{k-1}$ and $n-1-k$ in $W_{k+1}$ (since
there are exactly $k$ ones and $n-1-k$ zeros in the set
$\{x_1,x_2,\dots,x_{n-1}\}$). From Kirchhoff's and Ohm's Laws,
stating that the sum of all currents from a vertex is zero, and
the current that flows through an edge equals the difference in
voltages divided by the resistance of that edge (which are all one
in our case), we conclude that
$$
 k\times\frac{w_{k-1}-w_k}1+(n-1-k)\times\frac{w_{k+1}-w_k}1+ \frac{\tilde w_k-w_k}1 =0
$$
whence taking into account (\ref{eqtil})
\begin{equation}\label{eqk}
 w_k=\frac{k w_{k-1}+(n-k-1)w_{k+1}+1}{n+1}, \ k=1,2,\dots,n-2.
\end{equation}
Additionally, for $k=n-1$, we obtain in the same way
\begin{equation}\label{eqn-1}
 w_{n-1}=\frac{(n-1)w_{n-2}+1}{n+1}.
\end{equation}
Thus we have to solve the system of equations~(\ref{eq0}),
(\ref{eqk}), and~(\ref{eqn-1}). To this end, first set $w_k=\frac
12-u_k$ for all $k$, then our system becomes
\begin{eqnarray}\label{equ}
 u_0&=&1/2,\\
 u_k&=&\frac{k u_{k-1}+(n-k-1)u_{k+1}}{n+1}, \ k=1,2,\dots,n-1,
 \nonumber
\end{eqnarray}
with the additional condition $u_{n}=0$ (its value does not matter
anyway since it is multiplied by $0$ for $k=n-1$). Note that
intuitively we must end up with $u_k\ge 0$, since every vertex
$x\in W_k$ is closer to $a$ than to $b$. Thus the probability
$w_k$ to hit $b$ before $a$ should not exceed $\frac 12$.

We can rewrite system~(\ref{equ}) as
\begin{eqnarray}\label{equ1}
u_0&=&1/2, \\
u_{k-1}&=&\frac{{(n+1) u_{k}-(n-k-1)u_{k+1}}}{k}, \
k=1,2,\dots,n-1. \nonumber
\end{eqnarray}
Solving system~(\ref{equ1}) backwards, we obtain
\begin{eqnarray*}
u_{n-2} &=&\frac{n+1}{n-1} u_{n-1} ,\\
u_{n-3} &=&\frac{n^2+n+2}{(n-1)(n-2)} u_{n-1}\\
u_{n-4} &=&\frac{n^3+5n+6}{(n-1)(n-2)(n-3)} u_{n-1}
\end{eqnarray*}
%$$
%u_{n-5} =u_{n-1}\frac{n^4-2n^3+11n^2+14n+24}{(n-1)(n-2)(n-3)(n-4)}
%$$
etc. With some guessing, one can notice that
\begin{eqnarray}\label{eqind}
i u_{i-1}= n u_{n-1}+ (n-i) u_i
\end{eqnarray}
for $i=n,n-1,n-2,n-3$. Let us prove by induction that this is true
for all $i=1,\dots,n$. Indeed, we already know that~(\ref{eqind})
holds for $i=n,\dots,n-3$. Suppose that~(\ref{eqind}) holds for
$i=k+1,k+2,\dots,n$. Let us establish~(\ref{eqind}) for $i=k$.
Indeed, from~(\ref{equ1}), plugging in (\ref{eqind}) with $i=k+1$,
we obtain
$$
u_{k-1}=\frac{{(n+1) u_{k}-(n-k-1)u_{k+1}}}{k}
 =\frac{{(n+1) u_{k}-[(k+1) u_k-n u_{n-1}}]  }{k}
$$
yielding $k u_{k-1}= n u_{n-1}+ (n-k) u_k$ and thus completing the
induction.

In the next step we want to compute $u_i$ as a function of
$u_{n-1}$. Let us denote $u_{n-1}=c$ and substitute
$$
z_i=(i+1)(i+2)\dots (n-1) u_i
$$
into (\ref{eqind}). Then we have $z_{n-2}=(n+1) c$ and
$$
z_{i-1}=n(n-1)\dots(i+1)c+(n-i) z_i
$$
which after reiterations gives
 \begin{eqnarray*}
  \begin{array}{rclcl}
  z_{n-j-1}/c&=&n(n-1)\dots(n-j+1)&\times& 1
  \\&+& n(n-1)\dots(n-j+2)&\times &j \\
 &+& n(n-1)\dots(n-j+3)&\times&  j(j-1)\\
  &+&\dots\\
 &+&n(n-1)&\times& j(j-1)\dots 3 \\
 &+&n&\times& j(j-1)\dots 3\cdot 2 \\
 &+&1&\times& j(j-1)\dots 3\cdot 2\cdot 1,
 \end{array}
\end{eqnarray*}
that is,
\begin{eqnarray*}
  z_{n-j-1}&=&c\sum_{l=0}^{j} \frac{n!}{(n-j+l)!}\frac{j!}{(j-l)!}
   =cj!\sum_{l=0}^{j} {n \choose j-l}.
\end{eqnarray*}
Recalling that $u_0=1/2$ gives $z_0=(n-1)!/2$ whence
\begin{eqnarray*}
  \frac{(n-1)!}{2}&=&z_{0}
  =c(n-1)!\sum_{l=0}^{n-1} {n \choose n-1-l}
  =c(n-1)!\sum_{l=0}^{n-1} {n \choose l+1}\\
  &=&c(n-1)!\sum_{l=1}^{n} {n \choose l}
  =c(n-1)!\left(2^n-1\right).
\end{eqnarray*}
Therefore, $c=1/(2^{n+1}-2)$,
\begin{eqnarray*}
z_k=\frac{(n-1-k)!\sum_{l=0}^{n-1-k} {n \choose
n-1-k-l}}{2^{n+1}-2}
\end{eqnarray*}
and
\begin{eqnarray*}
u_k= \frac{k!}{(n-1)!}z_k
 =\frac{\sum_{l=0}^{n-1-k}{n \choose n-1-k-l}}{{n-1\choose k}\left(2^{n+1}-2\right)}
 =\frac{\sum_{i=k+1}^{n} {n \choose i}}{{n-1\choose
 k}\left(2^{n+1}-2\right)},
\end{eqnarray*}
Recalling that $w_k=1/2-u_k$ and hence $\tilde w_k=1/2+u_k$
finishes the proof.
\end{proof}
\begin{remark}
If $a=(0,0,\dots,0)$ and $b=(1,1,\dots,1)$, so that the points $a$
and $b$ are the furthermost points of the hypercube, one can
compute $p_{a,b}(x)$ very easily (we leave this as an exercise).
However,  with the exception of the two cases when $|a-b|=1$ and
$|a-b|=n$ we do not know a general formula for $p_{a,b}(x)$.
\end{remark}

\begin{remark}
On the other hand, the formula for the probability that the walk
started at $a=(0,0,\dots,0)$ is located at vertex
$x=(x_1,x_2,\dots,x_n)$ at time $N$ with $|x|=\sum_{i=1}^n x_i=k$
is known and given by formula~(3.1) in~\cite{Diac}:
\begin{eqnarray*}
\frac 1{2^n} \sum_{j=0}{n} \left[1- \frac{2j}{n+1}\right]^N
\sum_{i=0}^{k}(-1)^i {k \choose i}{n-k\choose j-i}.
\end{eqnarray*}
\end{remark}

\end{document}